\documentclass[12pt, onesided, reqno]{amsart}

\usepackage{amssymb}
\usepackage{amsmath}
\usepackage{color}
\usepackage{enumerate}
\usepackage{graphicx}

\evensidemargin\oddsidemargin

\newtheorem{theorem}{Theorem}
\newtheorem{proposition}[theorem]{Proposition}
\newtheorem{corollary}[theorem]{Corollary}
\newtheorem{lemma}[theorem]{Lemma}
\newtheorem{assumption}[theorem]{Assumption}

\theoremstyle{definition}

\newtheorem{remark}[theorem]{Remark}

\newcommand{\eqnsection}{
\renewcommand{\theequation}{\thesection.\arabic{equation}}
    \makeatletter
    \csname  @addtoreset\endcsname{equation}{section}
    \makeatother}
\eqnsection

\def\e{\mathbf{e}}
\def\E{\mathbb{E}}

\def\N{\mathbb{N}}

\def\R{\mathbb{R}}

\def\Pb{\mathbb{P}}
\def\C{\mathbb{C}}

\def\F{\mathcal{F}}

\newcommand{\equi}{\mathop{\sim}\limits}
\def\={{\,\;\mathop{=}\limits^{\text{(law)}}\;\,}}

\def\qed{\hfill$\square$}

\makeatletter
\@namedef{subjclassname@2020}{%
  \textup{2020} Mathematics Subject Classification}
\makeatother

\begin{document}

\title[]{Maximal displacement of spectrally negative\\ branching L\'evy processes}
\author[Christophe Profeta]{Christophe Profeta}

\address{
Universit\'e Paris-Saclay, CNRS, Univ Evry, Laboratoire de Math\'ematiques et Mod\'elisation d'Evry, 91037, Evry-Courcouronnes, France.
 {\em Email} : {\tt christophe.profeta@univ-evry.fr}
  }


\subjclass[2020]{}

\begin{abstract} 
We consider a branching Markov process in continuous time in which the particles evolve independently as spectrally negative L\'evy processes. When the branching mechanism is critical or subcritical, the process will eventually die and we may define its overall maximum, i.e. the maximum location ever reached by a particule. The purpose of this paper is to give asymptotic estimates for the survival function of this maximum. In particular, we show that in the critical case the asymptotics is polynomial when the underlying L\'evy process oscillates or drifts towards $+\infty$, and is exponential when it drifts towards $-\infty$.

\end{abstract}

\maketitle

\section{Introduction}

\subsection{Description of the model}

We consider a one-dimensional spectrally negative branching L\'evy process in the sense of Kyprianou \cite{Kyp1}. It is a continuous-time particle system in which individuals move according to independent spectrally negative L\'evy processes, and split at exponential times into a random number of children. \\

More precisely, an initial ancestor begins its existence at the origin at time $t=0$. It moves according to a spectrally negative L\'evy process $L$ up to an independent exponential random  variable $\e$ of parameter 1. It then dies and splits into a random number of children with distribution $\boldsymbol{p}=(p_n)_{n\geq0}$. Each of these children starts his life at the location of the ancestor and behaves independently of the others, following the same stochastic pattern as the ancestor : it moves according to $L$ and branches at rate 1.\\

We assume that the offspring distribution $\boldsymbol{p}$ is non trivial, has expectation smaller or equal to one  and admits moments of order at least 3 :
$$ p_1\neq1, \qquad  \qquad \E[\boldsymbol{p}]\leq 1,\qquad \qquad \E[\boldsymbol{p}^3]<+\infty.$$
As a consequence of the first two conditions, the branching process will almost surely die in finite time, and one may define its overall maximum $\bf{M}$, i.e. the maximum location ever reached by one particle. The purpose of this paper is to study the asymptotics of the survival function of 
$\bf{M}$.\\

The investigation of the maximal displacement of branching processes, or equivalently of their right-most particles, has already received a lot of attention in the literature. The emphasize has been put so far on the  supercritical branching Brownian motion for which it is known that the survival function of the right-most particle is a travelling wave solution of the F-KPP equation, see Bransom \cite{Bra}.\\

In our subcritical/critical set-up, this problem was first tackled by Fleischman \& Sawyer \cite{FlSa} in the case of the branching Brownian motion, as a model of propagation of an allele mutation in a population. In this case, one may write an ODE satisfied by the survival function of $\bf{M}$, and the result follows by standard analytic methods. A generalization to symmetric stable L\'evy processes was obtained by Lalley \& Shao \cite{LaSh1}, using a pseudo-differential equation and a Feynman-Kac representation of the solution.  More recently, the case of $\alpha$-stable L\'evy processes with positive jumps was solved in \cite{Pro}, using a different method based on integral equations. It was in particular proven that the asymptotics of the survival function of ${\bf M}$ is given as a power $-\alpha$ in the subcritical case, and $-\alpha/2$ in the critical case. We will see in the following theorems  that the situation is different for spectrally negative stable L\'evy processes.
 \\

It is finally noteworthy to point out that in the literature, the term "branching L\'evy process" might refer to a  construction more general than the one we just described. One may indeed encodes both the displacement of particules and the offspring reproduction into a general L\'evy measure : we refer for instance to Bertoin \& Mallein \cite{BeMa} or Mallein \& Shi \cite{MaSh} for a study of such processes.

\subsection{Statement of the results}

For $\lambda \in \C$ such that $\Re(\lambda)\geq0$,  let us define the Laplace exponent 
$\Psi(\lambda) = \ln \E\left[e^{\lambda L_1}\right] $ of the spectrally negative L\'evy process $L$ by

$$\Psi(\lambda) =a\lambda +  \frac{\eta^2}{2}\lambda^2 + \int_{-\infty}^0 \left(e^{\lambda x}- 1 - \lambda x 1_{\{|x|<1\}}\right) \nu(dx)$$
where $a \in \R$ is the drift coefficient, $\eta \in \R$ the Gaussian coefficient and the L\'evy measure $\nu$ satisfies $ \int_{-\infty}^0 (x^2\wedge 1)\, \nu(dx)<+\infty$. We exclude the case where $-L$ is a subordinator  (for which ${\bf M}=0$ a.s.). As a consequence the function $\Psi$ is strictly convex and tends to $+\infty$ as $\lambda\rightarrow +\infty$. This implies that for any $q\geq0$, the equation $\Psi(\lambda)=q$ admits at most two solutions, and we denote by $\Phi(q)$ the largest one :
$$\Phi(q) = \sup\{\lambda\geq 0,\, \Psi(\lambda)=q\}.$$
The function $\Phi$ is well-known to be related to the maximum of $L$. Indeed, let us denote by $S$ the running supremum of $L$ :
$$S_t = \sup_{s\leq t} L_s, \qquad t\geq0,$$
and let $\e$ be an exponential random variable with parameter 1 independent of $L$. Then, the random variable $S_\e$ is also exponentially 
distributed, see Bertoin \cite[Corollary VII.2]{Ber} : 
\begin{equation}\label{eq:S}
\Pb\left(S_\e \geq x\right) = e^{-\Phi(1) x}.
\end{equation}

\noindent
We start with the subcritical case.

\begin{theorem}\label{theo:subcrit}
Assume that the branching process is subcritical, i.e. $\E[\boldsymbol{p}]<1$. Then, there exists a finite constant $\kappa>0$ such that 
$$\Pb\left({\bf M}\geq x\right) \equi_{x\rightarrow +\infty}  \kappa e^{-\Phi(1-\E[\boldsymbol{p}])x}.$$

\end{theorem}

Comparing this asymptotics with (\ref{eq:S}), we see that in the subcritical case the branching mechanism reduces the decay of the survival function of $\bf{M}$, although it remains exponential. Of course, when there are no branching, i.e. $\E[\boldsymbol{p}]=0$, then ${\bf M} = S_\e$ a.s. and the equivalence is in fact an equality. \\

Before turning our attention to the critical case, we need to introduce the scale functions $W^{(q)}$, with $q\geq0$, which are defined on $[0,+\infty)$ via their Laplace transforms by 
\begin{equation}\label{eq:W}
\int_0^{+\infty} e^{-\beta x} W^{(q)} (x)dx = \frac{1}{\Psi(\beta)-q} \qquad \text{ for } \Re(\beta)>\Phi(q).
\end{equation}
These functions are increasing, a.e. differentiable, and known to be related with the exit time problem for spectrally negative L\'evy processes, see for instance Kuznetsov,  Kyprianou \& Rivero \cite{KKR} or Hubalek \& Kyprianou \cite{HuKy}. Analytically, the behavior at $+\infty$ of the scale function $W^{(0)}=W$ will explain the difference in the two cases of the following Theorem \ref{theo:crit}. Indeed, when $\Psi^\prime(0^+)<0$, the function $W^{(q)}$ has an exponential growth at $+\infty$ given for $q\geq0$ by 
\begin{equation}\label{eq:Wq}
W^{(q)}(x) \equi_{x\rightarrow +\infty} \frac{1}{\Psi^\prime(\Phi(q))} e^{\Phi(q) x}
\end{equation}
while in the case $\Psi^\prime(0^+)\geq 0$, this asymptotics is rather subexponential. To simplify the proof, we shall make the following assumption when $\Psi^\prime(0^+)\geq0$.
\begin{assumption}\label{assum}
The random variable $\max(-L_\e, 0)$ is integrable and the process $L$ is either
\begin{itemize}
\item of unbouded variation, 
\item or of bounded variation with a L\'evy measure having no atoms.
\end{itemize}
\end{assumption}

The first part of the assumption will ensure that the Fourier transforms we shall use in the proof are well-defined, while the second part implies that the scale function $W$ is of class $\mathcal{C}^1(0,+\infty)$, see \cite{HuKy}.

\begin{theorem}\label{theo:crit}
Assume that the branching process is critical, i.e. $\E[\boldsymbol{p}]=1$.
\begin{enumerate}
\item If $\Psi^\prime(0^+)> 0$ and Assumption \ref{assum} is satisfied, then :
$$\Pb\left({\bf M}\geq x\right) \equi_{x\rightarrow +\infty} \frac{2 \Psi^\prime(0^+) }{\sigma^2\, x} .$$
\item If $\Psi^\prime(0^+)= 0$  and Assumption \ref{assum} is satisfied, then there exist two positive constants $\kappa_1, \kappa_2$ such that for $x$ large enough :
$$\frac{\kappa_1}{xW(x)} \leq \Pb\left({\bf M}\geq x\right) \leq \frac{\kappa_2}{xW(x)}.$$
\item If $\Psi^\prime(0^+)<0$, then there exists a finite constant $\kappa>0$ such that :
$$\Pb\left({\bf M}\geq x\right) \equi_{x\rightarrow +\infty}  \kappa e^{-\Phi(0)x}.$$
\end{enumerate}

\end{theorem}
\medskip

\noindent
To better understand Theorem \ref{theo:crit}, it might be useful to recall the following facts on the large time behavior of the underlying L\'evy process $L$, see Bertoin \cite[Corollary VII.2]{Ber}  :
\begin{enumerate}[$i)$]
\item When $\Psi^\prime(0^+)>0$, the process $L$ drifts a.s. towards $+\infty$. As a consequence, each particule tends to drift upward, and the branching mechanism allows to obtain a power decay instead of an exponential one. Note that the inequalities given in Point (2) remain valid in this case since $\lim\limits_{x\rightarrow +\infty}W(x)=1/\Psi^\prime(0^+)$.
\item When $\Psi^{\prime}(0^+)=0$, the process $L$ oscillates, i.e. $\limsup\limits_{t\rightarrow +\infty} L_t = - \liminf\limits_{t\rightarrow +\infty} L_t  =+\infty$. In this case, the fact that some particules may enjoy big deviations towards $+\infty$ also yields a power decay. This is typically the case of the branching Brownian motion.
\item Finally, when $\Psi^{\prime}(0^+)<0$, the process $L$ drifts a.s. towards $-\infty$. This case is then very similar to the subcritical case of Theorem \ref{theo:subcrit}, i.e. the branching mechanism only slightly reduces the tail of the asymptotics.\\
\end{enumerate}

\noindent
When $\Psi^{\prime}(0^+)=0$, as is usual, Theorem \ref{theo:crit} may be refined by assuming a specific asymptotics of $\Psi$ at 0.

\begin{corollary}\label{cor:1}
Assume that $\E[\boldsymbol{p}]=1$ and  that $\Psi$ is regularly varying at 0, i.e.  that there exists a constant  $\alpha\in(1,2]$ and a slowly varying function $\ell$ such that
\begin{equation}\label{eq:condPsi}
\Psi(\lambda) \equi_{\lambda \downarrow 0} \lambda^{\alpha} \ell\left(\frac{1}{\lambda}\right).
\end{equation}
Then,
$$\kappa_1 \frac{ \ell(x)}{x^\alpha} \leq \Pb\left({\bf M}\geq x\right) \leq \kappa_2\frac{\ell(x)}{x^\alpha}$$
and there exists a sequence $(x_n)$ such that  
$$\Pb\left({\bf M}\geq x_n\right) \equi_{x_n\rightarrow +\infty} \frac{2}{\sigma^2} \frac{\Gamma(2\alpha)}{\Gamma(\alpha)  }\ell(x_n) x_n^{-\alpha}.$$
\end{corollary}

\begin{remark}
From Bertoin \cite[Proposition VII.6]{Ber}, Assumption (\ref{eq:condPsi}) is equivalent to the Spitzer's condition :
$$\lim_{t\rightarrow +\infty} \frac{1}{t}\int_0^t \Pb(L_s\geq 0) ds = \frac{1}{\alpha}.$$
The simplest example is of course the $\alpha$-stable spectrally negative L\'evy process for which one may choose $\Psi(\lambda) = \lambda^\alpha$, hence $\Phi(q) = q^{1/\alpha}$. Note that in this case, Assumption \ref{assum} is satisfied as $\alpha$-stable spectrally negative L\'evy processes are of unbounded variation, and admit moments of order 1.\\
\end{remark}

\subsection{An integral equation}

To prove Theorems \ref{theo:subcrit} and \ref{theo:crit}, we shall write down an integral equation which is similar to the stable case with positive jumps studied in \cite{Pro} or to the centered branching random walk case studied in Lalley \& Shao \cite{LaSh2}. Let us set 
$$u(x) = \Pb\left({\bf M}\geq x\right) \quad \text{for } x\geq0,$$
and $u(x)=0$ for $x<0$. The choice to take $u$ null on the negative half-line will allow to work with Fourier transforms without need of renormalization.

\begin{lemma}\label{lem:equ}
The function $u$ is a solution of the nonlinear integral equation :
\begin{multline}\label{eq:u}
u(x) = \E[\boldsymbol{p}]\E\left[1_{\{L_\e<x\}}u(x-L_\e)\right]- \frac{1}{2}\E\left[\boldsymbol{p}^2-\boldsymbol{p}\right]\E\left[1_{\{L_\e<x\}}u^2(x-L_\e)\right]\\
  +\E\left[1_{\{L_\e<x\}}R(x-L_\e)\right]+ \Delta(x)
\end{multline}
where the function $R$ is such that 
\begin{equation}\label{eq:omega3}
\forall z\in \R,\qquad 0\leq R(z) \leq \E[\boldsymbol{p}^3 ]u^3(z),
\end{equation}
and  the remainder $\Delta$ satisfies the bounds :
\begin{equation}\label{eq:phiR}
|\Delta(x)|  \leq \begin{cases}K  e^{-\Phi(1) x}  &\text{ if }x\geq0,\\
K\Pb(L_\e< x) & \text{ if }x<0
\end{cases}
\end{equation}
for some constant $K>0$.
\end{lemma}

\begin{proof}
Let $x\in \R$ and recall that $\e$ is an exponential random variable  of parameter 1 independent of $L$. We start by applying the Markov property at the first branching event~:
\begin{align*}
\Pb({\bf M}< x) &= p_0 \Pb\left(S_\e< x\right) + \sum_{n=1}^{+\infty} p_n\, \Pb\left(S_\e<x,\; L_\e+{\bf M}^{(1)}<x,\ldots,  L_\e+{\bf M}^{(n)}<x \right)
\end{align*}
where the  random variables $({\bf M}^{(n)})_{n\in \N}$ are independent copies of ${\bf M}$, and are also independent of the pair $(L_\e, S_\e)$. Using the formula
$\Pb({\bf M}\geq x) = u(x) +  1_{\{x<0\}}$
 we thus obtain the integral equation~:
\begin{equation}\label{eq:u0}
1-u(x)-1_{\{x<0\}} = p_0 \Pb\left(S_\e< x\right)+ \sum_{n=1}^{+\infty} p_n \, \E\left[1_{\{S_\e<x\}}\; (1- u(x-L_\e)-1_{\{x<L_\e\}})^n\right].
\end{equation}
Note that developing the power $n$ on the right-hand side, one may remove the indicator $1_{\{x<L_\e\}}$ since by definition $\{S_\e < x < L_\e\}=\emptyset$. 
Plugging the Taylor expansion with integral remainder 
$$(1-u)^n = 1 - n u + \frac{n(n-1)}{2} u^2 -  \frac{n(n-1)(n-2)}{6} u^3  \int_0^1 (1-u t)^{n-3} (1-t)^2 dt$$
in (\ref{eq:u0}), we deduce after some simplifications that 
\begin{multline}\label{eq:sansR}
u(x) =\Pb\left(S_\e\geq x\right) - 1_{\{x<0\}}
+  \E[\boldsymbol{p}]\E\left[1_{\{S_\e<x\}}u(x-L_\e)\right]\\-\frac{1}{2}\E\left[\boldsymbol{p}^2-\boldsymbol{p}\right]\E\left[1_{\{S_\e<x\}}u^2(x-L_\e)\right]+ \E\left[1_{\{S_\e<x\}}R(x-L_\e)\right]
\end{multline}
where the function $R$ equals :
$$R(z) = u^3(z) \sum_{n\geq 3}p_n   \frac{n(n-1)(n-2)}{6} \int_0^1   (1-u(z) t)^{n-3} (1-t)^2 dt  \,\leq\, \E[\boldsymbol{p}^3] u^3(z).$$
Looking at (\ref{eq:sansR}) and adding and substracting $1_{\{L_\e<x\}}$, one now obtains formula (\ref{eq:u}) with $\Delta$ given by 
\begin{multline*}
\Delta(x) = \Pb\left(S_\e\geq x\right) - 1_{\{x<0\}} \\+ \E\left[\left(1_{\{S_\e<x\}} - 1_{\{L_\e<x\}}\right)\left( \E[\boldsymbol{p}]u(x-L_\e)- \frac{1}{2}\E\left[\boldsymbol{p}^2-\boldsymbol{p}\right]u^2(x-L_\e)+R(x-L_\e)\right)\right].
\end{multline*}
Finally, for $x<0$, we have using that $S_\e\geq0$ a.s. and $u(z)\leq 1$ :
$$
|\Delta(x)|  \leq \E\left[1_{\{L_\e<x\}}\left( \E[\boldsymbol{p}]+\E\left[\boldsymbol{p}^2\right]+\E[\boldsymbol{p}^3]\right)\right] \leq K \Pb(L_\e<x)
$$
while for $x>0$, using the explicit distribution of $S_\e$ given by (\ref{eq:S}) :
$$
|\Delta(x)| \leq \Pb(S_\e\geq x) +  2 \E\left[1_{\{S_\e>x\}}\left( \E[\boldsymbol{p}]+\E\left[\boldsymbol{p}^2\right]+\E[\boldsymbol{p}^3]\right)\right]\\
\leq K  e^{-\Phi(1) x}.
$$
\end{proof}

\noindent
Starting from Lemma \ref{lem:equ}, the proofs of Theorems \ref{theo:subcrit} and \ref{theo:crit} both rely on the same three steps :
\begin{enumerate}[$i)$]
\item We  first obtain some a priori estimates on $u$ using Equation (\ref{eq:u}).
\item We then use these estimates to write down a new  integral equation satisfied by $u$.
\item We finally compute the asymptotics of $u$ using this new equation.
\end{enumerate}
One of the key observation will be to notice that the three expectations on the right-hand side of Equation (\ref{eq:u}) are in fact convolution products. This will lead us to work with Laplace and Fourier transforms.

\bigskip

\section{The subcritical case : proof of Theorem \ref{theo:subcrit}}\label{sec:2}

We start with the subcritical case and first prove that the asymptotics of $u$ must be at least exponential. In the following, we shall exclude the case $\E[\boldsymbol{p}]=0$ for which ${\bf M}= S_\e$ a.s.

\begin{lemma}\label{lem:p<1}
Assume that $\E[\boldsymbol{p}]<1$. There exists two constants $C, \delta>0$ such that 
$$\forall x\geq0,\qquad u(x) \leq C e^{-\delta x}.$$
\end{lemma}

\begin{proof} Using that $u$ is decreasing and bounded by $1$, we first write
\begin{equation}\label{eq:borneu2}
\Pb(L_\e>0)u^2(x)\leq \E\left[1_{\{L_\e>0\}}u^2(x-L_\e)\right] \leq  \E\left[1_{\{S_\e<x\}}u^2(x-L_\e)\right] + \Pb(S_\e\geq x).
\end{equation}
Of course, since $L$ is not the opposite of a subordinator, we necessarily have $\Pb(L_\e>0)>0$.
Going back to Equation (\ref{eq:sansR}), and using the bound on $R$ given in (\ref{eq:omega3}) as well as the obvious inequality $L_\e \leq S_\e$ a.s., we obtain for $x\geq0$ :
\begin{multline}\label{eq:startbound}
 \frac{1}{2}\E\left[\boldsymbol{p}^2-\boldsymbol{p}\right]\Pb(L_\e>0)u^2(x)\leq\E[\boldsymbol{p}]  \E\left[1_{\{S_\e<x\}}u(x-S_\e)\right]
 -u(x) \\+   \E[\boldsymbol{p}^3] \E\left[1_{\{S_\e<x\}}u^3(x-S_\e)\right] +\left(1+ \E\left[\boldsymbol{p}^2\right]\right)\Pb(S_\e\geq x).
\end{multline}
Notice that we have implicitly used the fact that $\E\left[\boldsymbol{p}^2-\boldsymbol{p}\right]>0$ since $\boldsymbol{p}$ is integer-valued. We now 
integrate this relation on $[0,n]$ with $n\in \N$ and $\delta>0$ :
\begin{align*}
&\int_0^n e^{\delta x}\left( \frac{1}{2}\E\left[\boldsymbol{p}^2-\boldsymbol{p}\right]\Pb(L_\e>0)u^2(x) -  \E[\boldsymbol{p}^3]  \E\left[1_{\{S_\e<x\}}u^3(x-S_\e)\right]\right) dx\\
&\qquad \leq  \E[\boldsymbol{p}]  \E\left[ e^{\delta S_\e}  \int_0^{(n-S_\e)^+}e^{\delta x} u(x)dx \right] - \int_0^n e^{\delta x} u(x)dx+ \left(1+\E\left[\boldsymbol{p}^2\right]\right)\int_0^n e^{\delta x}\Pb(S_\e\geq x) dx\\
&\qquad \leq  \left(\E[\boldsymbol{p}]  \E\left[ e^{\delta S_\e} \right]-1\right) \int_0^n e^{\delta x} u(x)dx+ \left(1+ \E\left[\boldsymbol{p}^2\right]\right)\int_0^n e^{\delta x}\Pb(S_\e\geq x) dx.
\end{align*}
Since $\E[\boldsymbol{p}]<1$, we may choose $\delta \in (0, \Phi(1))$ small enough such that $\E[\boldsymbol{p}]  \E\left[ e^{\delta S_\e} \right]<1$. Then, letting $n\rightarrow +\infty$, we obtain :
$$
\limsup_{n\rightarrow +\infty} \int_0^n e^{\delta x}\left( \frac{1}{2}\E\left[\boldsymbol{p}^2-\boldsymbol{p}\right]\Pb(L_\e>0)u^2(x) -  \E[\boldsymbol{p}^3]\E[e^{\delta S_\e}]u^3(x)\right) dx\leq \frac{1+\E[\boldsymbol{p}^2]}{\delta}\E[e^{\delta S_\e}].
$$
Since $ u$ decreases  to $0$ as $x\rightarrow +\infty$, we deduce that
$$\int_0^{+\infty} e^{\delta x}u^2(x)dx <+\infty$$
and finally, for $x\geq 1$,
\begin{equation}\label{eq:minu2}
\int_0^{+\infty} e^{\delta z}u^2(z)dz \geq \int_{x-1}^{x}e^{\delta z}u^2(z)dz \geq e^{\delta (x-1)} u^2(x)
\end{equation}
which implies the result.
\end{proof}

\bigskip

Thanks to Lemma \ref{lem:p<1}, we may now take the Laplace transform of Formula (\ref{eq:u}) with $\lambda \in (0,\delta)$. To simplify the computation, we set 
$$g(x) =    R(x)-\frac{1}{2}\E\left[\boldsymbol{p}^2-\boldsymbol{p}\right] u^2(x).$$
Using the Fubini theorem to compute the convolution products on the right-hand side of  (\ref{eq:u}), we obtain :
$$
 \int_\R e^{\lambda x} u(x) dx =  \E[\boldsymbol{p}] \E\left[e^{\lambda L_\e}\right] \int_\R e^{\lambda x} u(x) dx + 
 \E\left[e^{\lambda L_\e}\right] \int_\R e^{\lambda x} g(x) dx  +  \int_\R e^{\lambda x} \Delta(x) dx.
$$
By definition, the Laplace transform of the random variable $L_\e$ admits the expression
$$\E[e^{\lambda L_\e}] = \int_0^{+\infty}  e^{-t}  e^{t \Psi(\lambda )}dt =  \frac{1}{1-\Psi(\lambda)}$$
which yields the formula
\begin{equation}\label{eq:usubcrit}
\int_\R e^{\lambda x} u(x) dx  = \frac{ 1  }{1-\E[\boldsymbol{p}] - \Psi(\lambda)} \int_\R e^{\lambda x} (g(x)+\E[\boldsymbol{p}]\Delta(x)) dx  + \int_\R e^{\lambda x} \Delta(x) dx.  
\end{equation}
Observe that this expression remains in fact valid for $0\leq \lambda < \Phi(1-\E[\boldsymbol{p}])$.
We now prove that the fraction on the right-hand side may be written as a Laplace transform. Indeed 
\begin{align*}\frac{1}{1-\E[\boldsymbol{p}] - \Psi(\lambda)}  &= \int_0^{+\infty}  e^{-(1-\E[\boldsymbol{p}] - \Psi(\lambda)) t} dt \\
&=\int_0^{+\infty}  e^{-(1-\E[\boldsymbol{p}])t} \E\left[e^{\lambda L_t} \right]dt = \int_\R e^{\lambda z} \int_0^{+\infty} e^{-(1-\E[\boldsymbol{p}] )t} \Pb(L_t\in dz)dt.
\end{align*}
From Kyprianou \cite[Corollary 8.9]{Kyp}, the $q$-potential measure of $L$ is known to be absolutely continuous with density $\theta^{(q)}$ given for $q>0$ by :
\begin{equation}\label{eq:Kyp}
\int_0^{+\infty} e^{-qt}  \Pb(L_t\in dz)dt  =  \theta^{(q)}(z) dz =  \left( \Phi^\prime(q )e^{-\Phi(q)z} - W^{(q )}(-z)\right) dz.
\end{equation}
As a consequence, inverting Equation (\ref{eq:usubcrit}), we obtain a new integral equation for $u$ :
\begin{equation} \label{eq:u=g}
u(x) = \int_\R \left(g(z) +  \E[\boldsymbol{p}]\Delta(z)\right) \theta^{(1-\E[\boldsymbol{p}])}(x-z)\ dz + \Delta(x).
\end{equation}
To study the limit of $u$, observe first that from Formula (\ref{eq:omega3}), the function $g$ is ultimately negative. As a consequence, let us take $A$ large enough such that $g(x)\leq 0$ for $x\geq A$. Using the Definition (\ref{eq:Kyp}) of $\theta^{(q)}$, we have the upper bound
$$
u(x) 
\leq  \Phi^\prime(1-\E[\boldsymbol{p}] )  \int_0^Ag(z)  e^{\Phi(1-\E[\boldsymbol{p}] )(z-x)} dz +\E[\boldsymbol{p}]\int_\R  \Delta(z)  \theta^{(1-\E[\boldsymbol{p}])}(x-z)\ dz +  \Delta(x)
$$
where we have used that $g(x)$ and $W^{(q)}(x)$ are null for $x<0$.
Furthermore,  the second integral may be decomposed into
\begin{align*}
\int_\R  \Delta(z)  \theta^{(1-\E[\boldsymbol{p}])}(x-z)\ dz &= \Phi^\prime(1-\E[\boldsymbol{p}] )   \int_\R  \Delta(z)  e^{\Phi(1-\E[\boldsymbol{p}] )(x-z)} dz  - 
\int_x^{+\infty}   \Delta(z)  W^{(q )}(z-x) dx
\end{align*}
and, thanks to the bound (\ref{eq:phiR}) on $\Delta$ and the asymptotics (\ref{eq:Wq}),  
$$\int_x^{+\infty}   |\Delta(z)|  W^{(q )}(z-x) dx \leq \int_0^{+\infty}   |\Delta(z+x)|  W^{(q )}(z) dz \leq K    \int_0^{+\infty} e^{-\Phi(1) (z+x)} W^{(q)}(z) dz <+\infty.$$
As a consequence, taking the limit superior, we have proven that for $A$ large enough :
$$\limsup_{x\rightarrow +\infty} e^{\Phi(1-\E[\boldsymbol{p}] )x}u(x)  \leq  \Phi^\prime(1-\E[\boldsymbol{p}] )\int_\R  e^{\Phi(1-\E[\boldsymbol{p}] )z}  \left(g(z) 1_{\{0\leq z\leq A\}} +  \E[\boldsymbol{p}]\Delta(z)\right) dz.$$
The lower bound may be obtained similarly by observing that since $W^{(q)}$ is positive,
$$\int_A^{+\infty} g(z)  \theta^{(1-\E[\boldsymbol{p}])}(x-z)\ dz \geq \Phi^\prime(1-\E[\boldsymbol{p}] )   \int_A^{+\infty}g(z)  e^{\Phi(1-\E[\boldsymbol{p}] )(z-x)} dz$$
and is this time independent of $A$ :
$$\liminf_{x\rightarrow +\infty} e^{\Phi(1-\E[\boldsymbol{p}] )x}u(x)  \geq  \Phi^\prime(1-\E[\boldsymbol{p}] )\int_\R  e^{\Phi(1-\E[\boldsymbol{p}] )z}  \left(g(z) +  \E[\boldsymbol{p}]\Delta(z)\right) dz = \kappa.$$
Theorem \ref{theo:subcrit} now follows by letting $A\uparrow +\infty$.
%
%
It seems however difficult to compute explicitly $\kappa$ as it is given in terms of $u$.
\qed

\medskip

\section{The Critical case : Proof of Theorem \ref{theo:crit} when $\Psi^\prime(0^+)\geq0$}

We now assume that $\E[\boldsymbol{p}]=1$ and we recall that in this case $\sigma^2=\E[\boldsymbol{p}^2-\boldsymbol{p}]$ denotes the variance of $\boldsymbol{p}$. For the first part of the proof, we shall deal with Point (1) and Point (2) of Theorem \ref{theo:crit} simultaneously.

\begin{remark}\label{rem:subcrit}
Notice that by a coupling argument with the subcritical case obtained in Theorem \ref{theo:subcrit}, we deduce that for any $\delta>0$, 
\begin{equation}\label{eq:edu}
\liminf_{x\rightarrow +\infty} e^{\delta x}u(x) =+\infty.
\end{equation}
Indeed, starting from a critical spectrally negative branching process $X$ defined on a probability space $\Omega$, consider a second spectrally negative branching process $X^\ast$ defined on the same space $\Omega$ and where only the offspring distribution is modified as follows.
Fix $n_0\geq 1$ such that $p_{n_0}>0$ and let $\varepsilon>0$. Then define the new offspring distribution $\boldsymbol{p}^\ast$ by :
$$\forall \omega\in \Omega,\qquad \begin{cases}
\text{if } \boldsymbol{p}( \omega)  \neq n_0, \quad \text{then }  \boldsymbol{p}^\ast( \omega) = \boldsymbol{p}( \omega)\\
\text{if } \boldsymbol{p}( \omega)  = n_0,\quad \text{then }  \boldsymbol{p}^\ast( \omega) =  n_0 1_{\{\boldsymbol{U}( \omega)\leq 1-\varepsilon\}}  \\
\end{cases}
$$
where $\boldsymbol{U}$ is a uniform random variable on $[0,1]$ independent from $X$. In other words, if a particule has $n_0$ children, we remove, with probability $\varepsilon$, the paths of all these children. As a consequence, $X^\ast$ is a subcritical spectrally negative branching process since $\E\left[\boldsymbol{p}^\ast\right] = 1 - n_0 p_{n_0} \varepsilon <1$.
By coupling, we deduce with obvious notation that $u(x) \geq u^\ast(x)$, i.e., thanks to Theorem \ref{theo:subcrit}, 
$$\liminf_{x\rightarrow +\infty} e^{\Phi(1-\E[\boldsymbol{p}^\ast]) x}u(x) > \kappa^\ast>0.$$
This implies (\ref{eq:edu}) since $\Phi(0)=0$ in the case $\Psi^\prime(0^+)\geq0$.
\end{remark}

\subsection{A priori estimates}
As in the subcritical case, we start by proving some a priori estimates on  the function $u$. These estimates will be necessary to justify some of the computation later.
\begin{lemma}\label{lem:apriori+}
For any $\delta>0$, there exists a finite constant $C_\delta$ such that :
$$\forall x>0,\qquad u(x) \leq \frac{C_\delta}{x^{1-\delta}}. $$
\end{lemma}

\begin{proof}
We shall prove by induction that for every $n\in \N$, there exists a constant $C_n$ such that 
\begin{equation}\label{eq:rec}
\forall x>0,\qquad  u(x) \leq  C_n\,x^{\frac{1}{2n}-1}.
\end{equation}
\subsubsection{The base case $n=1$.} 
We start by taking the Laplace transform of Equation (\ref{eq:startbound}), whose right-hand side only involves the random variable $S_\e$. This yields with $\lambda>0$ :
$$
\Pb(L_\e>0)\frac{\sigma^2}{2} \int_0^{+\infty}e^{-\lambda z}u^2(z) dz - \E\left[e^{-\lambda S_\e}\right] \E[\boldsymbol{p}^3]\int_0^{+\infty}e^{-\lambda z}u^3(z) dz \leq     \left(1+\E[\boldsymbol{p}^2]\right)\frac{1-\E\left[e^{-\lambda S_\e}\right]}{\lambda}$$
i.e. using the explicit distribution of $S_\e$ 
$$
\int_0^{+\infty}e^{-\lambda z}\left(\Pb(L_\e>0)\frac{\sigma^2}{2} u^2(z) - \E[\boldsymbol{p}^3] u^3(z)\right)dz\le \frac{1+\E[\boldsymbol{p}^2]}{\lambda+ \Phi(1)}.$$
As before, since $u$ is decreasing and converges towards 0 as $x\rightarrow +\infty$, we deduce by letting $\lambda\downarrow0$ that $u^2$ is integrable. Then 
a change of variables yields
$$u^2\left(\frac{1}{\lambda}\right)\int_0^{1}e^{- z}dz  \leq  \lambda \int_0^{+\infty} e^{-\lambda z}u^2 (z) dz\leq \lambda \int_0^{+\infty} u^2 (z) dz.$$
Setting $\lambda =1/x$, we finally conclude that there exists a finite constant $C_{1}$ such that :
$$u(x) \leq \frac{C_{1} }{\sqrt{x}}$$
which is (\ref{eq:rec}) for $n=1$.
\subsubsection{Induction step}
Fix $n\in \N$ and assume that Formula (\ref{eq:rec}) holds true. Multiplying Equation (\ref{eq:startbound}) by $x$ and taking the Laplace transform, we obtain after some simplifications using the decomposition $x=x-S_\e + S_\e$ :
\begin{align*}
&\int_0^{+\infty}e^{-\lambda x}x\left(\Pb(L_\e>0)\frac{\sigma^2}{2} u^2(x) - \E[\boldsymbol{p}^3] u^3(x)\right)dx\\
&\qquad \leq \E\left[S_\e e^{-\lambda S_\e }\right] \int_0^{+\infty}e^{-\lambda x}u(x)dx + \frac{1 +\E[\boldsymbol{p}^2] }{(\lambda + \Phi(1))^2} +    \E[\boldsymbol{p}^3] \E\left[S_\e e^{-\lambda S_\e }\right] \int_0^{+\infty}e^{-\lambda x}u^3(x)dx.
\end{align*}
Fix  $\varepsilon$ small enough and take $A>0$  such that for any $x\geq A$ 
$$  \E[\boldsymbol{p}^3]  u(x) \leq \Pb(L_\e>0)\frac{\sigma^2}{2} - \varepsilon. $$
A change of variables then yields
\begin{multline*}
\frac{\varepsilon}{\lambda^2} \int_{\lambda A}^{+\infty} e^{-z} z u^2\left(\frac{z}{\lambda}\right)dz \\\leq
\frac{\E\left[S_\e\right]}{\lambda} \int_0^{+\infty}e^{-\lambda z}\left(\frac{z}{\lambda}\right)dz + \frac{1 +\E[\boldsymbol{p}^2] }{(\lambda + \Phi(1))^2} +    \E[\boldsymbol{p}^3] \E\left[S_\e\right] \int_0^{+\infty}u^3(z)dz +  \E[\boldsymbol{p}^3] \int_0^A z u^3(z)dz
\end{multline*}
i.e., for $\lambda$ small enough such that $\lambda A\leq 1/2$, 
$$\varepsilon u^2\left(\frac{1}{\lambda}\right) \int_{1/2}^{1} e^{-z} z dz \leq \lambda \E\left[S_\e\right] \int_0^{+\infty}e^{- z}u\left(\frac{z}{\lambda}\right) dz + \lambda^2K$$
for some constant $K>0$.
Plugging the induction hypothesis in the right-hand side yields
$$\varepsilon u^2\left(\frac{1}{\lambda}\right) \int_{1/2}^{1} e^{-z} z dz \leq \lambda^{2-\frac{1}{2^n}} C_n\E\left[S_\e\right] \int_0^{+\infty}e^{- z} z^{1-\frac{1}{2n}}dz + \lambda^2K$$
and replacing $\lambda=1/x$ as before, we finally obtain the existence of a constant $C_{n+1}$ such that~:
$$u(x)   \leq C_{n+1}\, x^{\frac{1}{2^{n+1}}-1}.$$
As a consequence, we deduce that  (\ref{eq:rec}) holds for every $n\in \N$, which proves Lemma \ref{lem:apriori+} by monotony.

\end{proof}

\subsection{A new equation for $u$}
The purpose of this subsection is to prove the following new equation for the function $u$ :

\begin{proposition}
The function $u$ is a solution of the integral equation :
$$u(x)-\Delta(x) = \int_0^{+\infty} \left(\frac{\sigma^2}{2} u^2(x+z)- R(x+z)-\Delta(x+z)\right)W(z) dz.$$
\end{proposition}

\noindent
In the following, we shall denote the  Fourier transform of a measurable function $f$, provided it exists, by :
$$\F(f)(\xi) = \lim_{n\rightarrow +\infty} \int_{-n}^n e^{i\xi z} f(z) dz=  \int_\R e^{i\xi z} f(z) dz.$$
In particular, since $u$ is decreasing and converges towards 0, we deduce from the Abel-Dirichlet's convergence test for improper integrals that $\F(u)$ is well-defined (although $u$ might not be integrable).  Taking the Fourier transform of Equation (\ref{eq:u}), and applying the Fubini theorem on the left-hand side since $u^2$ is integrable from Lemma \ref{lem:apriori+}, we deduce that  
\begin{equation}\label{eq:F-}
\frac{\sigma^2}{2} \E[e^{i\xi L_\e}] \F(u^2)(\xi) = \int_\R e^{i\xi x}  \int_\R u(x-z) \Pb(L_\e\in dz) dx-\F(u)(\xi) + \E[e^{i\xi L_\e}]  \F(R)(\xi) +\F(\Delta)(\xi). 
\end{equation}
Note that the Fourier transform of $\Delta$ is well-defined, since from Lemma \ref{lem:equ} and Assumption \ref{assum} the function $\Delta$ is integrable, i.e. 
$$\int_\R |\Delta(x)| dx\leq   K \int_0^{+\infty} e^{-\Phi(1) x} dx + K \int_{-\infty}^{0} \Pb(L_\e< x)dx = \frac{K}{\Phi(1)} + K \E[\max(-L_\e, 0)] <+\infty.$$

\noindent
To compute the remaining convolution product on the right-hand side, we shall rely on the following lemma :

\begin{lemma}\label{lem:Fourier}
Let $f: \R\rightarrow \R$ be an integrable function and let $\varphi : [0,+\infty)\rightarrow [0,+\infty)$ be a decreasing  function converging towards 0.
%
Then, the Fourier transform of the convolution product is given by : 
$$  \int_\R  e^{i\xi x} \left( \int_\R \varphi(x-z) f(z) dz\right) dx = \F(\varphi)(\xi) \times \F(f)(\xi),\qquad \quad\xi\neq 0.$$
\end{lemma}

\begin{proof}
Of course, Lemma \ref{lem:Fourier} is just a consequence of Fubini theorem if $\varphi$ is integrable. But when $\varphi$ is not integrable, as will be the case here, some care is needed.
In particular, observe first that thanks to the Abel-Dirichlet's convergence test for improper integral, the Fourier transform of $\varphi$ is well-defined for $\xi\neq0$ :
$$\forall \xi\neq0,\qquad \left|\int_0^{+\infty}e^{i\xi x} \varphi(x) dx  \right|<+\infty.$$
Now, fix $n>0$. Since, by assumption, 
$$\int_{-n}^n \int_0^{+\infty} \left| e^{i\xi x}  f(x-z) \varphi(z)\right| dz \,dx \leq 2n \varphi(0) \int_{-\infty}^{+\infty} |f(x)|dx <+\infty$$
we may apply the Fubini theorem and a change of variable to obtain 
$$\int_{-n}^n e^{i\xi x} \int_0^{+\infty} f(x-z) \varphi(z)dz \,dx = \int_0^{+\infty} e^{i\xi z} \varphi(z)\int_{-n -z}^{n-z} e^{i\xi y} f(y) dy\, dz.$$
To avoid lengthy expressions, we shall proceed in two steps by cutting the last integral at $0$.
Integrating by parts, we first write:
\begin{align*}
& \F(\varphi)(\xi) \times \int_0^{+\infty} e^{i\xi y} f(y) dy
 - \int_0^{+\infty} e^{i\xi z} \varphi(z)\int_{0}^{n-z} e^{i\xi y} f(y) dy\, dz \\
& = \int_0^{+\infty} e^{i\xi z} \varphi(z) dz \int_{n}^{+\infty} e^{i\xi y} f(y) dy  + \int_0^{+\infty}   \left(\int_y^{+\infty} e^{i\xi z} \varphi(z) dz\right)  e^{i\xi(n-y)} f(n-y) dy.
\end{align*}
The first term will converge to 0 as $n\rightarrow +\infty$. To study the limit of the second term, let $\varepsilon>0$ and fix $A$ large enough such that 
$$\forall x\geq A,\qquad \left|\int_x^{+\infty} e^{i \xi z} \varphi(z)dz \right| \leq \varepsilon.$$
We then have :
\begin{align*}
&\left|\int_0^{A}   \left(\int_y^{+\infty} e^{i\xi z} \varphi(z) dz\right)  e^{i\xi(n-y)} f(n-y) dy\right| + \left| \int_A^{+\infty}   \left(\int_y^{+\infty} e^{i\xi z} \varphi(z) dz\right)  e^{i\xi(n-y)} f(n-y) dy \right|\\
&\leq     \sup_{y\geq0}\left|\int_y^{+\infty} e^{i\xi z} \varphi(z) dz\right| \int_0^A|f(n-y)|dy +  \varepsilon \int_\R  |f(x)| dx \xrightarrow[n\rightarrow +\infty]{}  \varepsilon \int_\R  |f(x)| dx,
\end{align*}
which proves that
$$\lim_{n\rightarrow +\infty} \int_0^{+\infty} e^{i\xi z} \varphi(z)\int_{0}^{n-z} e^{i\xi y} f(y) dy\, dz =  \F(\varphi)(\xi) \times \int_0^{+\infty} e^{i\xi y} f(y) dy.$$
A similar argument shows that
$$ 
 \lim_{n\rightarrow+\infty} \int_0^{+\infty} e^{i\xi z} \varphi(z)\int_{-n-z}^0 e^{i\xi y} f(y) dy\, dz =\F(\varphi)(\xi) \times \int_{-\infty}^{0} e^{i\xi y} f(y) dy. $$
As a consequence, we have obtained that 
\begin{align*}
\int_\R e^{i\xi x} \int_0^{+\infty} f(x-z) \varphi(z)dz \,dx &= \lim_{n\rightarrow +\infty} \int_{-n}^n e^{i\xi x} \int_0^{+\infty} f(x-z) \varphi(z)dz \,dx \\
&=\lim_{n\rightarrow +\infty}  \int_0^{+\infty} e^{i\xi x} \varphi(x)\int_{-n -x}^{n-x} e^{i\xi z} f(z) dz\, dx \\
&=  \F(\varphi)(\xi) \times \F(f)(\xi)
\end{align*}
which concludes the proof of Lemma \ref{lem:Fourier}.
\end{proof}

\medskip
\noindent
Recall now that the distribution of $L_\e$ is actually the $1$-potential measure of $L$, i.e. from (\ref{eq:Kyp}) the random variable $L_\e$ is absolutely continuous with density given by $\theta^{(1)}$. Applying Lemma \ref{lem:Fourier} to Equation (\ref{eq:F-}) with $f=\theta^{(1)}$ which is integrable, and $\varphi=u$ which is decreasing and converging towards 0, we thus obtain
$$
\frac{\sigma^2}{2} \E[e^{i\xi L_\e}] \F(u^2)(\xi) = \E[e^{i\xi L_\e}] \F(u)(\xi) -\F(u)(\xi) +\E[e^{i\xi L_\e}] \F(R)(\xi)+ \F(\Delta)(\xi). 
$$
By definition, the characteristic function of $L_\e$ is given by 
$$\E\left[e^{i\xi L_\e}\right] = \int_0^{+\infty}  e^{-t}  e^{t \Psi(i\xi)}dt =  \frac{1}{1-\Psi(i \xi)}$$
which yields the equation
$$\F\left(\frac{\sigma^2}{2} u^2- R-\Delta\right)(\xi)  = \Psi(i\xi)  \F(u-\Delta)(\xi),  $$
or equivalently, 
\begin{equation}\label{eq:u'}
\frac{i\xi}{\Psi(i\xi)} \F\left(\frac{\sigma^2}{2} u^2-R-\Delta\right)(\xi) = i\xi \,\F(u-\Delta)(\xi).
\end{equation}


\noindent
The next step consists in showing that the function $\displaystyle \xi\longrightarrow i\xi/\Psi(i\xi)$ is actually the Fourier transform of $W^\prime$, which is well-defined thanks to Assumption \ref{assum}.
Integrating by parts the definition of $W$ given in (\ref{eq:W}) we obtain
\begin{equation}\label{eq:W'}
\int_0^{+\infty} e^{-\beta x} W^\prime(x)dx = \frac{\beta}{\Psi(\beta)},\qquad \Re(\beta)>0.
\end{equation}
We now write down the extension of Formula (\ref{eq:W'}) to the case $\beta = i\xi$ with $\xi\neq0$. 
Since $W$ is increasing and concave, the function $W^\prime$ is positive and decreasing, and we shall denote by $w_\infty$ its limit : $\lim\limits_{x\rightarrow +\infty} W^\prime(x) = \inf\limits_{x\geq 0}W^\prime(x)= w_\infty$. As a consequence, from the Abel- Dirichlet's convergence test for improper integrals, the Fourier transform of $W^\prime-w_\infty$ is well-defined for $\xi\neq 0$.
Let $\varepsilon>0$ and take $A$ large enough such that 
$$\forall x\geq A,\qquad \left|\int_x^{+\infty} e^{-i \xi z} (W^\prime(z)-w_\infty)dz \right| \leq \varepsilon.$$
 Integrating by parts, we have for $h>0$ :
\begin{align*}
\left|\int_0^{+\infty} \left(e^{- hx}-1\right) e^{-i\xi x} (W^\prime(x)-w_\infty)dx  \right| &=h  \left|\int_0^{+\infty}e^{- hx} \int_x^{+\infty }e^{-i\xi z} (W^\prime(z)-w_\infty)dz\, dx  \right| \\
&\leq h \int_0^A \left|  \int_x^{+\infty }e^{-i\xi z} (W^\prime(z)-w_\infty)dz\right|dx + \varepsilon
\end{align*}
which proves that 
$$ \lim_{h\downarrow 0} \int_0^{+\infty} e^{- hx} e^{-i\xi x} (W^\prime(x)-w_\infty)dx =\int_0^{+\infty} e^{-i\xi x} (W^\prime(x)-w_\infty)dx.$$
Then, using the continuity of $\Psi$, we obtain still for $\xi\neq0$ :
\begin{align*}
\notag \frac{i\xi}{\Psi(i\xi)}   &= \lim_{h\downarrow 0}\frac{h+i\xi}{\Psi(h+i\xi)}\\
\notag &=  \lim_{h\downarrow 0} \int_0^{+\infty} e^{- hx} e^{-i\xi x} (W^\prime(x)-w_\infty)dx +  \int_0^{+\infty} e^{- hx} e^{-i\xi x} w_\infty dx   \\
& =\int_0^{+\infty} e^{-i\xi x} (W^\prime(x)-w_\infty)dx + \frac{w_\infty}{i\xi}.
\end{align*}

Plugging this last expression  in (\ref{eq:u'}) and computing the convolution product, using again Lemma \ref{lem:Fourier} with $f =\frac{\sigma^2}{2} u^2-R- \Delta$ which is integrable thanks to Lemma \ref{lem:apriori+}, and $\varphi=W^\prime-w_\infty$ which is decreasing with limit 0, one obtains :
\begin{multline}\label{eq:FF}
\int_{-\infty}^{+\infty} e^{i \xi x} \int_0^{+\infty} \left(\frac{\sigma^2}{2} u^2(z+x)-R(z+x)- \Delta(z+x)\right)(W^\prime(z)-w_\infty) dz \, dx \\
 =i\xi\, \F(u-\Delta)(\xi) - \frac{w_\infty}{i\xi}\F\left(\frac{\sigma^2}{2} u^2-R- \Delta\right)(\xi) .
\end{multline}
We now check that the terms in $w_\infty$ cancel. When $\Psi^\prime(0^+)>0$, we necessarily have $w_\infty=0$ since the function $W^\prime$ is integrable as can be seen by letting $\beta\downarrow 0$ in  Formula (\ref{eq:W'}) and applying the monotone convergence theorem
$$\int_0^{\infty} W^\prime(x) dx = \frac{1}{\Psi^\prime(0^+)}.$$
When $\Psi^\prime(0^+)=0$, the cancelation will follow from the observation that 
$$  \F\left(\frac{\sigma^2}{2} u^2-R- \Delta\right)(0) =\int_{-\infty}^{+\infty}  \left(\frac{\sigma^2}{2} u^2(z)-R(z)- \Delta(z)\right) dz=0.$$
Indeed, applying the dominated convergence theorem in (\ref{eq:u'})  thanks to Lemma \ref{lem:apriori+}, we have 
\begin{align*}
\left| \F\left(\frac{\sigma^2}{2} u^2-R- \Delta\right)(0)\right|
&=\lim_{\xi\downarrow 0} \left| \Psi(i\xi)  \F\left(u-\Delta\right)(\xi)\right|\\
& \leq \lim_{\xi\downarrow 0} |\Psi(i\xi)|\left(\frac{2}{\xi}  + \int_\R |\Delta(x) |dx \right)  =0 . 
\end{align*}
As a consequence, integrating by parts the last term of  (\ref{eq:FF}), we obtain :
$$\frac{w_\infty}{i\xi}\F\left(\frac{\sigma^2}{2} u^2-R- \Delta\right)(\xi)  =  w_\infty  \int_\R e^{i\xi x} \int_{x}^{+\infty}\left( \frac{\sigma^2}{2} u^2(z)-R(z)- \Delta(z)\right) dz dx.$$
 Similarly, using a change of variable, the left-hand side of  (\ref{eq:FF}) equals 
 $$\int_{-\infty}^{+\infty} e^{i \xi x} \int_x^{+\infty} \left(\frac{\sigma^2}{2} u^2(y)-R(y)- \Delta(y)\right)(W^\prime(y-x)-w_\infty) dy\, dx $$
i.e., Equation (\ref{eq:FF}) reduces to 
$$\int_{-\infty}^{+\infty} e^{i \xi x} \int_x^{+\infty} \left(\frac{\sigma^2}{2} u^2(y)-R(y)- \Delta(y)\right)W^\prime(y-x)dy\, dx 
 =i\xi\, \F(u-\Delta)(\xi) .
$$
It now remains to integrate by parts the left-hand side, and then inverse the Fourier transform to obtain :
\begin{equation}\label{eq:avant}
 \int_x^{+\infty}  \int_r^{+\infty}  \left(\frac{\sigma^2}{2} u^2(z)-R(z)- \Delta(z)\right)W^\prime(z-r)dz  dr = u(x)-\Delta(x).
\end{equation}
Finally, we deduce from Remark \ref{rem:subcrit} and the bounds on $R$ and $\Delta$ that for $x$ large enough, there exists a constant $K>0$ such that 
\begin{equation}\label{eq:eq>0}
\frac{\sigma^2}{2} u^2(x)-R(x)-\Delta(x) \geq \frac{\sigma^2}{2} u^2(x)\left( 1 -  K u(x)- K e^{-(\Phi(1) - \delta) x}\right)\geq 0 
\end{equation}
which proves that the integrand on the right-hand side of (\ref{eq:avant}) is positive for $x$ large enough. As a consequence, applying the Fubini-Tonelli theorem, we finally obtain 
\begin{align*}
u(x)-\Delta(x) & = \int_x^{+\infty}  \left(\frac{\sigma^2}{2} u^2(z)-R(z)- \Delta(z)\right) W(z-x) dz\\
& = \int_0^{+\infty} \left(\frac{\sigma^2}{2} u^2(z+x)-R(z+x)- \Delta(z+x)\right) W(z) dz
\end{align*}
which is the announced equation.\qed
%
%
%
%
%

\subsection{Study of the limit}

The last part of the proof now consists in studying the asymptotics of $u$, using the new equation :
\begin{equation}\label{eq:ux1z}
u(x)-\Delta(x)= x  \int_0^{+\infty} \left(\frac{\sigma^2}{2} u^2(x(z+1))-R(x(z+1))-\Delta(x(z+1))\right) W(xz) dz .
\end{equation}
Going back to Theorem \ref{theo:crit},  Point (2) is equivalent to showing that 
\begin{equation}\label{eq:infsup}
0< \liminf_{x\rightarrow +\infty} \gamma(x) \leq  \limsup_{x\rightarrow +\infty}\gamma(x)  <+\infty
\end{equation}
where we have set, to simplify the notation :
$$\gamma(x) = xW(x)u(x).$$

\subsubsection{Computation of the upper bound}
We first prove that the limit superior of $\gamma$ is finite.
As in Equation (\ref{eq:eq>0}), let us take $A$ large enough such that for any $x\geq A$, the quantity $\frac{\sigma^2}{2}u^2(x)-R(x)-\Delta(x)$ is positive. We then decompose :
\begin{align*}
 u(x)&\geq  x\int_1^{2}  W(xz) \left( \frac{\sigma^2}{2}u^2(x(z+1)) -R(x(z+1))-\Delta(x(z+1))\right) dz  \\
  &\geq  \frac{\sigma^2}{2} xW(x) u^2(3x) -  xW(x)\E[\boldsymbol{p}^3]u^3(2x)  - xW(x) K e^{-2\Phi(1)x} 
\end{align*}
for some constant $K>0$ given by the bound (\ref{eq:phiR}) on $\Delta$.
Then, multiplying both sides by $xW(x)$ and taking the supremum on $[A,n]$, we deduce that 
$$
\sup_{[A, n]}   \frac{\sigma^2}{18} \left(\frac{W(x)}{W(3x)}\right)^2 \left(\gamma(3x)\right)^2 
\leq \sup_{[A, n]} \gamma(x) +\frac{\E[\boldsymbol{p}^3] u(2A)}{4}  \sup_{[A, n]} \left(\frac{W(x)}{W(2x)}\right)^2  \left(\gamma(2x)\right)^2 + C
$$
for some constant $C>0$ independent of $n$. Furthermore, since $W$ is positive,  increasing and concave, we have the bounds 
$$1 \geq \frac{W(x)}{ W(2x)}\geq  \frac{W(x)}{W(3x)} \geq   \frac{1}{W(3x)} \left(\int_0^{x}W^\prime(3y) dy +W(0)\right) = \frac{1}{3}+ \frac{2}{3}\frac{W(0)}{W(3x)}\geq \frac{1}{3}.$$
Therefore
$$ \frac{\sigma^2}{162} \sup_{[3A, 3n]}     \left(\gamma(x)\right)^2 \leq  \sup_{[A, 3n]} \gamma(x)  + \frac{ \E[\boldsymbol{p}^3]}{4} u(2A)  \sup_{[A, 3n]} \left(\gamma(x)\right)^2 + C 
$$
and dividing both sides 
$$
\left( \frac{\sigma^2}{162}-  \frac{ \E[\boldsymbol{p}^3]}{4} u(2A) \right)  \sup_{[3A, 3n]} \gamma(x) \leq  \left(1+\frac{\sup_{[A, 3A]} \gamma(x) +C}{ \sup_{[3A, 3n]}     \gamma(x)} \right)   + \frac{ \E[\boldsymbol{p}^3]}{4} u(2A)  \frac{\sup_{[A, 3A]} \left(\gamma(x)\right)^2}{ \sup_{[3A, 3n]}    \gamma(x)}. 
$$
Finally, by taking $A$ large enough for the left-hand side to be positive and letting $n\rightarrow +\infty$, we conclude that 
$$\sup_{x\geq 3A} \gamma(x) <+\infty$$
which implies that the limit superior is finite. \\

\subsubsection{Computation of the lower bound}
We now turn our attention to the limit inferior of $\gamma$.
Let us fix $\delta>0$ small enough and start by writing the decomposition 
$$u^2(z)\leq u^{1-\delta}(z) (u(z) - \Delta(z))^{1+\delta} \left(1+ \frac{\Delta(z)}{u(z)-\Delta(z)}\right)^{1+\delta}.$$
Using the limit superior as well as (\ref{eq:edu}), we deduce that there exists $K>0$ and $A>0$ such that for every $z\geq A$, one has 
$$\frac{\sigma^2}{2}u^2(z)- \Delta(z) \leq K \left(\frac{1}{zW(z)}\right)^{1-\delta}(u(z)-\Delta(z))^{1+\delta} .$$ 
Furthermore, looking at Formula (\ref{eq:avant}), we see that the function $u-\Delta$ is differentiable, and decreasing for $x$ large enough since $W^\prime$ is positive. This allows to obtain the bound :
$$
\left(\Delta(x) - u(x)\right)^\prime \leq  K   (u(x)-\Delta(x))^{1+\delta}   \int_0^{+\infty} \frac{xW^\prime(xz)  }{\left(x(1+z)W(x(1+z)\right)^{1-\delta}}dz 
$$
which implies, since $x\rightarrow xW(x)$ is increasing, 
\begin{equation}\label{eq:u'/u}
\frac{\left(\Delta(x) - u(x)\right)^\prime }{(u(x)-\Delta(x))^{1+\delta} }  \leq  K \left(\frac{W(x)}{(xW(x))^{1-\delta}} +   x\int_1^{+\infty} \frac{W^\prime(xz)  }{\left(x(1+z)W(x(1+z)\right)^{1-\delta}}dz \right). 
\end{equation}
Furthermore, since $W$ is concave and $z\geq 1$, we have the series of inequalities :
$$
\frac{W^\prime(xz) }{\left(x(1+z)W(x(1+z))\right)^{1-\delta}} \leq \frac{W(xz)}{xz}   \frac{1}{\left(xzW(xz)\right)^{1-\delta}}
= \frac{\left(W(xz)\right)^\delta}{(xz)^{2-\delta}} \leq   \left(\frac{W(x)}{x}\right)^{\delta}\frac{1}{(xz)^{2-2\delta}}.$$
As a consequence, we deduce that the integral in (\ref{eq:u'/u}) is bounded by 
$$ x\int_1^{+\infty} \frac{W^\prime(xz) }{\left(x(1+z)W(x(1+z))\right)^{1-\delta}} dz \leq  \left(\frac{W(x)}{x}\right)^{\delta} \int_x^{+\infty} \frac{1}{z^{2-2\delta}}dz  =   \frac{1}{1-2\delta} \frac{(W(x))^\delta}{x^{1-\delta}}.
$$
Then, integrating (\ref{eq:u'/u}) on $[A, y]$, the last inequality implies that there exists a constant $C>0$ such that  :
$$
\frac{1}{(u(y)-\Delta(y))^\delta} -\frac{1}{(u(A)-\Delta(A))^\delta} \\
\leq 
\delta C  \int_A^y \frac{(W(x))^\delta}{x^{1-\delta}}dx \leq  C(yW(y))^{\delta}
$$
i.e.
$$
u(y)-\Delta(y) \geq \left( C(yW(y))^{\delta}   + \frac{1}{(u(A)-\Delta(A))^\delta}   \right)^{-1/\delta}.
$$
Finally, passing to the limit inferior, we conclude that 
$$\liminf_{y\rightarrow +\infty} \gamma(y)= \liminf_{y\rightarrow +\infty} yW(y) u(y) \geq C^{-1/\delta}>0$$
which proves Point (2) of Theorem \ref{theo:crit}.\\

\subsection{The case $\Psi^\prime(0^+)>0$.}
We now prove Point (1) of Theorem \ref{theo:crit}.
In this case, we may improve the previous inequality using the fact that $W$ converges towards some positive value $W(\infty) = 1/\Psi^\prime(0^+)$. Indeed, let us first write Equation (\ref{eq:ux1z}) under the form :
\begin{equation}\label{eq:alpha1}
u(x)= \frac{\sigma^2}{2}   \int_0^{+\infty}  u^2(x+z) W(x+z)dz  + I(x) = \frac{\sigma^2}{2}   \int_x^{+\infty}  u^2(z) W(z)dz  + I(x)
\end{equation}
where the remainder $I$ is given by :
$$I(x) =  \Delta(x) + \frac{\sigma^2}{2}   \int_0^{+\infty}  u^2(x+z)(W(z)-W(z+x)) dz  -  \frac{\sigma^2}{2}   \int_0^{+\infty} (R(z+x) + \Delta(z+x)) W(z) dz. $$
Differentiating and solving Equation (\ref{eq:alpha1}), we deduce that 
\begin{align}
\notag u(x) &= \frac{1}{1+ \frac{\sigma^2}{2}   \int_0^x W(z)dz  + \int_0^x \frac{I^\prime(z)}{u^2(z)}dz }\\
\label{eq:u=I}&= \frac{1}{1+ \frac{\sigma^2}{2}   \int_0^x W(z)dz  +  \frac{I(x)}{u^2(x)} - I(0) -2 \int_0^x \frac{I(z)u^\prime(z)}{u^3(z)}dz }.
\end{align}
Now, using (\ref{eq:infsup}) as well as the bounds on $R$ and $\Delta$, the function $I$ is smaller than 
$$
|I(x)| \leq K_1\left(\frac{1}{x^2} + \frac{1}{x} \int_0^{+\infty} \frac{1}{(1+z)^2} (W(x+zx)-W(zx)) dz\right) = K_1\left(\frac{1}{x^2} + \frac{\rho(x)}{x}\right)$$
for some constant $K_1>0$ and where $\rho$ is a positive function converging towards 0.
Next, for $x\geq 1$, since from the first part of the proof the function $x\rightarrow xu(x)$ is bounded away from 0, we obtain the bound 
$$ \int_1^x \left|\frac{I(z)u^\prime(z)}{u^3(z)}\right| dz \leq  K_2 \left(1+   |\ln(u(x))| + \int_1^x  \frac{\rho(z)}{u^2(z)} |u^\prime(z)| dz\right)
$$
for some constant $K_2>0$. To deal with the last integral, fix $\varepsilon>0$. There exists $A>0$ such that for $x\geq A$,  $\rho(x)<\varepsilon$. This allows to obtain the decomposition 
$$ \int_1^x  \frac{\rho(z)}{u^2(z)} |u^\prime(z)| dz \leq   \int_1^A  \frac{\rho(z)}{u^2(z)} |u^\prime(z)| dz + \varepsilon \left(\frac{1}{u(x)}-\frac{1}{u(A)} \right).$$
%
As a consequence, multiplying Equation  (\ref{eq:u=I}) by $x$ and passing to the limits as $x\rightarrow +\infty$, we deduce that 
$$\frac{1}{ \frac{\sigma^2}{2} W(\infty) + C \varepsilon  }    \leq \liminf_{x\rightarrow +\infty} xu(x) \leq\limsup_{x\rightarrow +\infty} xu(x)  \leq \frac{1}{ \frac{\sigma^2}{2} W(\infty) - C \varepsilon}$$ 
for some constant $C>0$, and the result follows by letting $\varepsilon\downarrow 0$ :
$$\lim_{x\rightarrow +\infty} xu(x) = \frac{1}{\frac{\sigma^2}{2} W(\infty)} = \frac{2}{\sigma^2} \Psi^\prime(0^+).$$\qed

\subsection{The regularly varying case}
We now assume that Assumption (\ref{eq:condPsi}) holds :
$$
\Psi(\lambda) \equi_{\lambda \downarrow 0} \lambda^{\alpha} \ell\left(\frac{1}{\lambda}\right)
$$
where $\alpha\in(1,2]$ and $\ell$ is a slowly varying function.
Recalling that $W$ is increasing, we deduce from (\ref{eq:W}) and Karamata's Tauberian theorem that $W$ is regularly varying at $+\infty $ :
\begin{equation}\label{eq:WW'}
W(x) \equi_{x\rightarrow +\infty}  \frac{x^{\alpha-1}}{\Gamma(\alpha)\ell\left(x\right)}.
\end{equation}
Let us define for $z\geq 0$ the function $b^{(x)}$ by 
$$b^{(x)}(z) =\frac{W(x)W(xz)}{ \left((1+z) W(x+xz)\right)^2}. $$
From the asymptotics (\ref{eq:WW'}), this function converges a.s. : 
$$b^{(x)}(z) \xrightarrow[x\rightarrow +\infty]{}  \frac{z^{\alpha-1}}{(1+z)^{2\alpha}}. $$
We now come back to Equation (\ref{eq:ux1z}) which we rewrite :
\begin{equation}\label{eq:gammaI}
\gamma(x)=    \frac{\sigma^2}{2}   \int_0^{+\infty}\gamma^2(x(1+z)) b^{(x)}(z)dz+  I(x)
\end{equation}
where the remainder $I(x)$ is given by
$$
I(x) = xW(x)\Delta(x)-\int_0^{+\infty} \big(x(1+z) W(x(1+z))\big)^2\big(R(x(1+z)) + \Delta(x(1+z))\big) b^{(x)}(z) dz.
$$
Since $W$ is increasing, the inequality $b^{(x)}(z) \leq 1/(1+z)^2$ together with the bounds on $R$ and $\Delta$ given in Lemma \ref{lem:equ} allow to apply the  dominated convergence theorem to obtain 
$$\lim_{x\rightarrow +\infty } I(x) =0.$$
Now, applying Fatou's lemma in Equation (\ref{eq:gammaI}), we deduce that 
$$\liminf_{x\rightarrow +\infty} \gamma(x) \geq \liminf_{x\rightarrow +\infty} \gamma^2(x)  \frac{\sigma^2}{2} \int_0^{+\infty} \frac{z^{\alpha-1}}{(1+z)^{2\alpha}}dz $$
i.e.
$$\liminf_{x\rightarrow +\infty} \gamma(x) \leq  \frac{2}{\sigma^2B(\alpha,\alpha)} $$
where $B$ denotes the Beta function. Similarly, applying the reverse Fatou lemma since $\gamma$ is bounded, we conclude that 
$$ \frac{2}{\sigma^2B(\alpha,\alpha)} \leq \limsup_{x\rightarrow +\infty} \gamma(x) . $$
Finally, the existence of a sequence $(x_n)$ such that 
$$\lim_{n\rightarrow +\infty} \gamma(x_n) = \lim_{x\rightarrow +\infty} x_nW(x_n) u(x_n) =  \frac{2}{\sigma^2B(\alpha,\alpha)}$$
is a consequence of the continuity of $\gamma$. Now, Corollary \ref{cor:1} follows from the asymptotics of $W$ given by (\ref{eq:WW'}) and the formula for the Beta function $\displaystyle B(\alpha,\alpha)=\Gamma^2(\alpha)/\Gamma(2\alpha)$.\qed \\

\begin{remark}
Notice that if we neglect the remainders, the equation 
$$f(x)=    \frac{\sigma^2}{2}   \int_0^{+\infty}f^2(x(1+z)) \frac{z^{\alpha-1}}{(1+z)^{2\alpha}}  dz$$
admits as solutions the functions 
$$f_c(x) = \frac{x^\alpha}{\left(c + \left(\frac{\sigma^2}{2}B(\alpha,\alpha)\right)^{1/\alpha}x\right)^\alpha}$$
where $c$ is any parameter in $[0,+\infty]$. This expression is in agreement with the explicit solution obtained in \cite{FlSa} for the Brownian case.
\end{remark}

\section{The critical case : proof of Theorem \ref{theo:crit} when $\Psi^\prime(0^+)<0$}

In this case, the underlying L\'evy process drifts a.s. to $-\infty$, and we shall see that the asymptotics of $\bf{M}$ is no longer polynomial but exponential. Consequently, the proof will be similar to the subcritical case of Theorem \ref{theo:subcrit}. We start with the following a priori estimate :

\begin{lemma}\label{lem:psi-}
Assume that $\Psi^\prime(0^+)<0$. Then, for any $\delta >0$, there exists a finite constant $C_\delta$ such that 
$$\forall x\geq 0,\qquad u(x) \leq C_\delta \exp\left(-\frac{\Phi(0)-\delta}{2} x\right).$$

\end{lemma}

\begin{proof}
Fix $\lambda\in(0, \Phi(0))$ and notice that since $\Psi$ is convex and  $\Psi^\prime(0^+)<0$, we have $\E\left[e^{\lambda L_\e}\right]\leq 1$.
Let $n\in \N$. Integrating Equation (\ref{eq:u}) against $\exp\left(\lambda x - \frac{1}{n}e^{\lambda x}\right)$, we obtain :
\begin{align*}
 &\int_0^{+\infty}\E\left[e^{\lambda (x+L_\e)}  \exp\left(- \frac{1}{n}e^{\lambda (x+L_\e)}\right)\right]\left(\frac{\sigma^2}{2}u^2(x) - \E[\boldsymbol{p}^3]u^3(x) \right) dx -  \int_0^{+\infty} e^{\lambda x}  \exp\left(- \frac{1}{n}e^{\lambda x}\right) \Delta(x) dx\\
&\qquad \leq \int_0^{+\infty}\E\left[e^{\lambda (x+L_\e)}  \exp\left(- \frac{1}{n}e^{\lambda (x+L_\e)}\right)\right]u(x)dx - \int_0^{+\infty} e^{\lambda x}  \exp\left(- \frac{1}{n}e^{\lambda x}\right) u(x)dx. 
\end{align*}
Integrating by parts and recalling that $u$ is decreasing, the right-hand side is further equal to 
$$\frac{n}{\lambda}\E\left[e^{-\frac{1}{n}e^{\lambda L_\e}} - 1\right] + \int_0^{+\infty} \E\left[e^{-\frac{1}{n}e^{\lambda x}} -e^{-\frac{1}{n}e^{\lambda x + \lambda L_\e}}  \right] |u^\prime(x)| dx.
$$
Observe finally that this last integral is negative, since, from Jensen inequality,
$$
\E\left[e^{-\frac{1}{n}e^{\lambda x}} -e^{-\frac{1}{n}e^{\lambda x + \lambda L_\e}}  \right] \leq e^{-\frac{1}{n}e^{\lambda x}} -e^{-\frac{1}{n}e^{\lambda x} \E\left[e^{ \lambda L_\e}\right]}\leq 0.
$$
As a consequence, we deduce that 
$$\limsup_{n\rightarrow +\infty} \int_0^{+\infty}\E\left[e^{\lambda (x+L_\e)}  \exp\left(- \frac{1}{n}e^{\lambda (x+L_\e)}\right)\right]\left(\frac{\sigma^2}{2}u^2(x) - \E[\boldsymbol{p}^3]u^3(x) \right) dx \leq   \int_0^{+\infty} e^{\lambda x} \Delta(x) dx.$$
Applying the monotone convergence theorem, this implies that
$$\int_0^{+\infty}    e^{\lambda x } u^2(x) dx<+\infty.$$
Lemma \ref{lem:psi-} now follows from the fact that $u$ is decreasing, using the same argument as in (\ref{eq:minu2}).
\end{proof}

The remainder of the proof is now similar to the subcritical case of Section \ref{sec:2}. Taking the Laplace transform of (\ref{eq:u}) with $\lambda \in (0, \Phi(0)/2)$ , we obtain :
\begin{equation}\label{eq:-1/psi}
 -\frac{1}{\Psi(\lambda)} \int_\R e^{\lambda x}  \left(\frac{\sigma^2}{2} u^2(x) - R(x)- \Delta(x)\right)dx  = \int_\R e^{\lambda x}  \left(u(x)- \Delta(x)\right)dx .
 \end{equation}
 We now observe that since $\Psi(\lambda)<0$, the function $\lambda\longrightarrow -1/\Psi(\lambda)$ is the Laplace transform of the $0$-potential of $L$ :
$$-\frac{1}{\Psi(\lambda)} = \int_0^{+\infty} e^{t\Psi(\lambda)}dt = \int_0^{+\infty} \E\left[e^{\lambda L_t}\right] dt= \int_\R e^{\lambda z} \int_0^{+\infty} \Pb(L_t\in dz)  dt. $$
Therefore,  passing to the limit as $q\downarrow 0$ in (\ref{eq:Kyp}) 
$$
\int_0^{+\infty} \Pb(L_t\in dz)dt  = \left( \Phi^\prime(0)e^{-\Phi(0)z} - W(-z)\right) dz
$$
and plugging this last relation in (\ref{eq:-1/psi}), we obtain after inverting the Laplace transforms
$$\int_\R \left(\frac{\sigma^2}{2}u^2(z)- R(z)- \Delta(z)\right)  \left( \Phi^\prime(0)e^{-\Phi(0)(x-z)} - W(z-x)\right) dz= u(x)-\Delta(x).$$
Finally, passing to the limit as $x \rightarrow +\infty$, we conclude that
$$\lim_{x\rightarrow +\infty} e^{\Phi(0) x} u(x) =  \Phi^\prime(0)  \int_\R e^{\Phi(0) z}\left(\frac{\sigma^2}{2}u^2(z) - R(z)- \Delta(z)\right) dz=\kappa. $$
\qed

\end{document}